\documentclass{article}
\usepackage{graphicx} 

\usepackage{paralist,amsmath,amsfonts,amsthm,amssymb,upgreek,dsfont,mathrsfs,textcomp,svrsymbols,slashed}

\usepackage{enumerate, mathtools, xcolor, cite, tikz, verbatim, soul, float, hhline} 

\usepackage{url}

\usepackage{subfig}

\usepackage[a4paper,margin=1in,footskip=0.25in]{geometry}

\title{Vertex orders in higher dimensions}
\author{Bennet Goeckner\thanks{University of San Diego} \ and Marta Pavelka\thanks{Carnegie Mellon University}}

\definecolor{FoundersBlue}{RGB}{0,59,112} 

\newtheorem{theorem}{Theorem}[section]
\newtheorem{corollary}[theorem]{Corollary}

\theoremstyle{definition}
\newtheorem{definition}[theorem]{Definition}

\newtheorem{example}[theorem]{Example}

\newcommand\set[1]{\left\{ #1 \right\}}
\newcommand\abs[1]{\left| #1 \right|}

\newcommand\ideal[1]{\langle #1 \rangle}

\DeclareMathOperator{\link}{link}
\DeclareMathOperator{\del}{del}

\begin{document}

\maketitle

\begin{abstract}
Unit interval and interval complexes are higher-dimensional generalizations of unit interval and interval graphs, respectively. We show that strongly connected unit interval complexes are shellable with shellings induced by their unit interval orders. We also show that these complexes are vertex decomposable and hence shelling completable. On the other hand, we give simple examples of strongly connected interval complexes that are not shellable in dimensions two and higher.
\end{abstract}
\section{Introduction}

Many classes of graphs can be defined via an ordering on their vertices; two such well-known examples are unit interval and interval graphs. In this paper, we study higher-dimensional generalizations of these objects. A (unit) interval graph is the intersection graph of a set of (unit) intervals. Looges and Olario \cite[Theorem 1, Proposition 4]{LO93} characterized these graphs using labelings on their vertices. 
In \cite{BSV}, authors extend the idea of vertex labelings to higher dimensions and define unit interval complexes as a generalization of unit interval graphs, and interval complexes as a generalization of interval graphs. The relationship between these two classes extends from graphs: unit interval complexes are interval. An even weaker class of complexes, semi-closed complexes, is naturally defined in this context even though it does not correspond to a previously known graph class. 

The significance of these classes lies, for example, in their relationship to the existence of Hamiltonian cycles in graphs and simplicial complexes. In the 1980s, Bertossi showed \cite{Bertossi} that if connected, then the unit interval graphs admit a Hamiltonian path, i.e. a path that touches all vertices exactly once. In the 1990s, Chen--Chang--Chang proved \cite{CCC} that 2-connected unit interval graphs admit a Hamiltonian cycle. In \cite[Chapters 1-2]{BSV} the authors generalize both results to simplicial complexes. 

Moreover, the above-mentioned generalizations of graph classes inherit algebraic properties as well. Herzog et al. \cite{HHHKR} first encoded graphs into binomial edge ideals in 2010 and showed that all such binomial edge ideals turn out to be radical. Benedetti, Seccia, and Varbaro \cite{BSV} showed that the determinantal facet ideals of all semi-closed complexes are also radical. 

Shellable complexes are important and extensively studied in combinatorics with connections to topology, geometry, and algebra. 
Pure shellable complexes have particularly nice algebraic and topological properties. For example, their face rings are Cohen--Macaulay and their geometric realizations are homotopy equivalent to wedges of top-dimensional spheres. However, deciding whether a pure $2$-dimensional simplicial complex is shellable is NP-complete as shown by Goaoc et al. \cite{GPPTW}. 

The property of being lex shellable naturally bridges shellability and vertex orders: A shelling is a lex shelling if it is induced by the lexicographic order on the vertices. Matroid independence complexes can be characterized as the pure complexes for which every order of their vertices induces a lex shelling \cite[Theorem~7.3.4]{Bj92}. Though lex shellability is comparatively less well understood than shellability in general, some recent progress has been made. For example, Sentinelli showed that if a complex $\Delta$ has a lex shelling given some vertex order, then any linear extension of the Bruhat order on the facets of $\Delta$ (subject to the same vertex order) is a shelling order \cite[Theorem~3.2]{Se24}.

A stronger property than shellability is vertex decomposability, first introduced in \cite{PB80}. Vertex decomposability implies shellability and shelling completability (introduced and proved in \cite{CDGO}), and all matroid independence complexes are vertex decomposable. In contrast, vertex decomposability turns out to be incomparable with being lex shellable. 

In this paper, we connect the two worlds: relatively new properties defined using vertex orders with classical combinatorial properties. We show that a unit interval order on the vertices of a pure strongly connected simplicial complex induces a lex shelling of this complex (Theorem \ref{thm:unitintervallex}). Moreover, we show that a strongly connected unit interval complex is vertex decomposable and hence shelling completable (Theorem \ref{thm:unitintervaldecomposable}, Corollary \ref{cor:unitintervalcompletable}). For interval orders, the same is true only in dimension one; in higher dimensions, we construct examples of non-shellable complexes that admit interval orders (Theorem \ref{thm:underclosed}). Lastly, a semi-closed order does not induce lex shelling even in graphs (Example \ref{ex:semi_closed_not_lex}). On the other hand, a complex can have lex shelling orders and be vertex decomposable without satisfying any of the aforementioned vertex orders (Example \ref{ex:matroid_not_semi_closed}). 

\section{Preliminaries}

Additional details on simplicial complexes can be found in, for example, \cite{Bj95}. A \textit{simplicial complex} $\Delta$ is a family of sets that is closed under taking subsets. These sets are called \textit{faces} of $\Delta$, and the dimension of a face $E$ is $|E|-1$. The dimension of $\Delta$ is $\max _{E\in\Delta} \dim E$. The top-dimensional faces are called \textit{facets} of $\Delta$, and a simplicial complex is called \textit{pure} if all its facets have the same dimension. The \textit{dual graph} of a pure simplicial complex is obtained by placing a node for each facet and connecting two nodes by an edge if and only if the corresponding facets share a codimension one face. A simplicial complex is \textit{strongly connected} if its dual graph is connected. We recall the definitions of the \textit{link} of $v$, denoted $\link_\Delta(v)$, and the \textit{deletion} of $v$, denoted $\del_\Delta(v)$, of a vertex $v$ in a simplicial complex $\Delta$:
\begin{equation*}
    \link_\Delta(v) = \{E\in\Delta ~|~ v\notin E \text{ and } v\cup E \in \Delta\}  \qquad \del_\Delta(v) = \{ E \in \Delta ~|~ v\notin E\}
\end{equation*}

Throughout the paper, $d$ and $n$ are positive integers and $d<n$. Recall that $[k] = \set{1,\dots,k}$ for any positive integer $k$. We describe each face $E \in \Delta$ as a list of its vertices, for example $E = v_0 v_1 \dots v_k$ where we assume $v_0 < v_1 < \dots < v_k$ unless otherwise stated. Moreover, we describe a simplicial complex as a list of its facets, for example $\Delta = \ideal{F_1,\dots,F_k}$ where the $F_i$ are the facets.

Simplicial complexes in dimension one are simple graphs. \textit{(Unit) interval} graphs are intersection graphs of sets of (unit) intervals (see Figures \ref{fig:intgraph}, \ref{fig:unitint}). These classical graphs have many equivalent definitions, including their description with vertex orders by Looges and Olario \cite{LO93}. A graph $G$ is \emph{interval} if and only if it admits a vertex labeling such that for all $a<b<c$: 
\begin{equation*}
    \text{edge } ac \in G \quad \Longrightarrow \quad \text{edge } ab\in G \ .
\end{equation*}
Such a labeling corresponds to reading intervals from left to right as in Figure \ref{fig:intgraph}. A graph $G$ is \emph{unit interval} if and only if it admits a vertex labeling such that for all $a<b<c$: 
\begin{equation*}
    \text{edge } ac \in G \quad \Longrightarrow \quad \text{edges } ab \text{ and } bc\in G\ .
\end{equation*}

\begin{figure}[hbt] 
    \centering
		\subfloat[]{\includegraphics[height=10em]{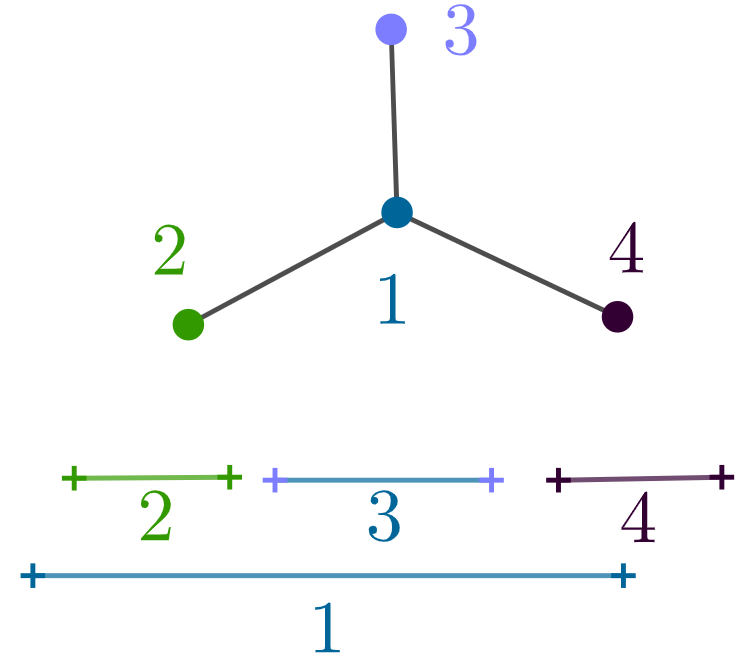}\label{fig:intgraph}}
		\hspace{1.5in}
		\subfloat[]{\includegraphics[height=10em]{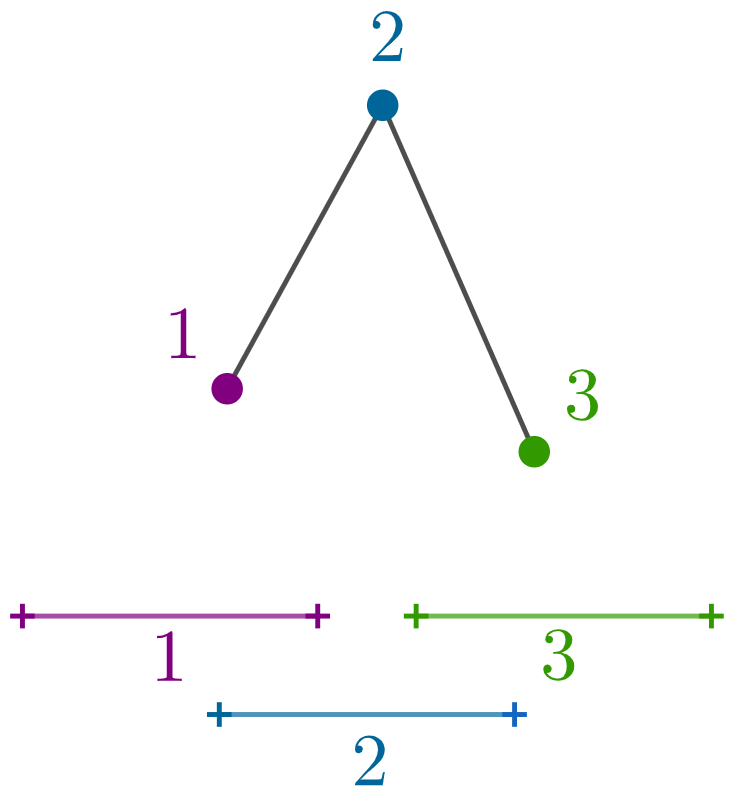}\label{fig:unitint}}
	\caption{(a) A claw graph that is an interval graph but not unit interval graph. (b) A unit interval graph.}
  \label{fig:graphs}
\end{figure}

These vertex order descriptions of (unit) interval graphs gave rise to the following generalizations to simplicial complexes introduced in \cite{BSV}.
\begin{definition}
Let $\Delta$ be a pure $d$-dimensional simplicial complex on $n$ vertices. 
\begin{itemize}
    \item Assume that if $v_0 v_1 \dots v_d$ is a facet of $\Delta$, then $\Delta$ contains the complete $d$-skeleton on $\set{v_0, v_0 + 1 , \dots, v_d}.$ Then this ordering of the vertices is called a \emph{unit interval order}, and any complex with such an ordering is a \emph{unit interval complex}.
    \item Assume that if $v_0 v_1 \dots v_d$ is a facet of $\Delta$, then $\Delta$ contains every facet $v_0 w_1 \dots w_d$, where $w_i \le v_i$ for all $i \in [d].$ Then this ordering of the vertices is called an \emph{interval order}, and any complex with such an ordering is an \emph{interval complex}.
    \item  Assume that if $v_0 v_1 \dots v_d$ is a facet of $\Delta$, then $\Delta$ contains         \begin{compactenum}[(i)]
             \item either all facets $v_0 w_1 \dots w_d$, where $w_i \le v_i$ for all $i \in [d]$,
            \item or all facets $w_0 w_1 \dots w_{d-1} v_d$, where $w_i \geq v_i$ for all $i \in \{0, 1, \dots, d-1\}.$
        \end{compactenum}
    Then this ordering of the vertices is called a \emph{semi-closed order}, and any complex with such an ordering is a \emph{semi-closed complex}.
\end{itemize}
\end{definition}

\begin{figure}[hbt] 
    \centering
		\subfloat[]{\includegraphics[height=10em]{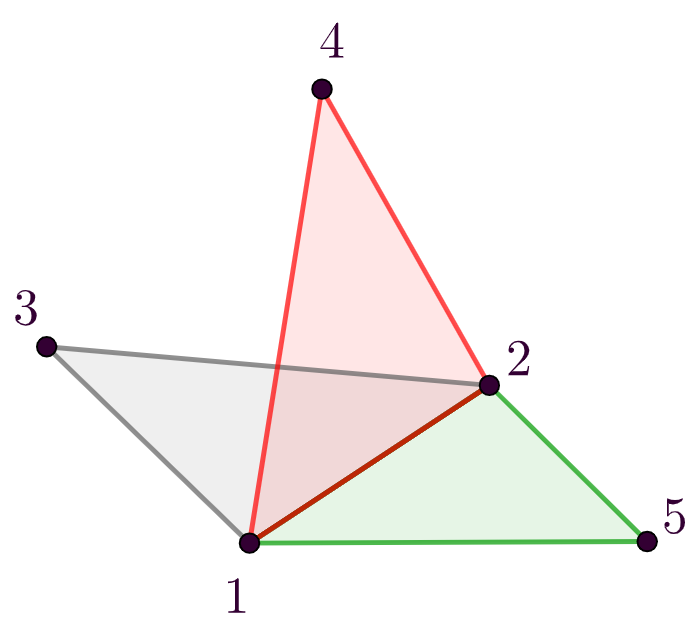}
		\label{fig:undercomplex}}
		\hspace{1.5in}
		\subfloat[]{\includegraphics[height=10em]{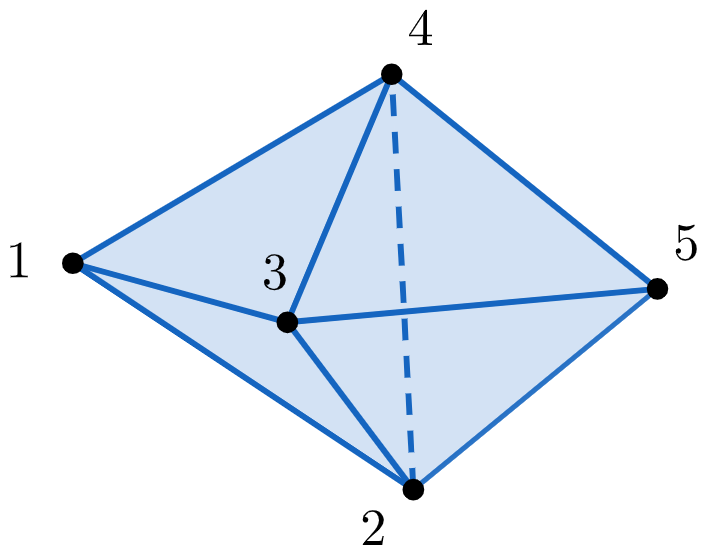}
		\label{fig:unitcomplex}}
		\caption{(a) An interval complex that is not unit interval. (b) A semi-closed complex that is not interval: 123, 124, 134, 235, 245, 345 \cite{BSV}.}
\end{figure}

We note that interval complexes were called \emph{under-closed} complexes in \cite{BSV}. In dimension one, unit interval complexes are precisely unit interval graphs, and interval complexes are precisely interval graphs. Semi-closed complexes, on the other hand, boil down to a natural class of graphs that has not been studied before it was introduced in \cite{BSV}. Furthermore, for strongly connected complexes, the hierarchy depicted in equation \eqref{eq:hierarchy} strictly holds in dimensions greater than zero (Figures \ref{fig:undercomplex}, \ref{fig:unitcomplex}); for undefined terms, see \cite[Section 2.2]{BSV}. 
When not strongly connected, the closed property is incomparable with the others listed in the hierarchy below \cite[Lemma 43]{BSV}; all other relationships still hold.
\begin{equation} \label{eq:hierarchy}
    \{ \text{ closed }\} \subsetneq  \{ \text{ unit interval }\} \subsetneq  \{ \text{ interval }\} \subsetneq  \{ \text{ semi-closed }\} \subsetneq  \{ \text{ weakly-closed }\}
\end{equation}

Closed complexes were defined by Ene et al. \cite{EHHM13} as a generalization of unit interval graphs in the context of generalizing binomial edge ideals. When strongly connected, all closed complexes are unit interval \cite[Proposition 54]{BSV}. The other classes in the above hierarchy were defined in \cite{BSV}.

\section{Shellability}

Shellability occupies a central place in the study of simplicial complexes.

\begin{definition}
    A \emph{shelling} of a pure $d$-dimensional simplicial complex $\Delta$ is an order on its facets $F_1 , \dots , F_m$ such that $\ideal{F_{k+1}} \cap \ideal{F_1,\dots,F_k}$ is pure and $(d-1)$-dimensional for all $k \in [m-1]$. If a shelling is induced by the lexicographic order on the vertices of the complex, we call it a \emph{lex shelling}. If a complex has a (lex) shelling, then it is \emph{(lex) shellable}.
\end{definition}

\begin{example}\label{ex:simple_lex_shell}
    Let $\Delta = \ideal{123,234,345}$ and notice that the given order on the vertices is a unit interval order. The order $123,234,345$ is a shelling order and thus is a lex shelling. Though $234,123,345$ is also a shelling order of $\Delta$, it is not a lex order under this labeling of the vertices.
\end{example}

In dimension one, pure shellability is equivalent to connectedness, and every connected $1$-dimensional complex has a lex shelling (simply label any vertex as $1$ and then label subsequent vertices so that they are adjacent to at least one vertex with a lower label). However, this is not the case in higher dimensions, as shown in the following example due to Doolittle, Lazar, and the first author \cite{DGL}.

\begin{example}\label{ex:VD_no_lex_shellings}

Let $\Delta$ be the following complex on vertex set $\{0,1,2,\dots,9\}.$
$$
\Delta = \ideal{123, 129, 135, 029, 234, 124, 136, 127, 027, 017, 289, 278, 346, 345, 089, 245, 789, 018, 156, 256}
$$
The given order of the facets above is a shelling. However, an exhaustive computer search has verified that no ordering on the vertices of this complex induces a lex shelling. 

\end{example}

Unit interval complexes are not matroid independence complexes in general---Example~\ref{ex:simple_lex_shell} is not a matroid because some orders on the vertices do not induce a shelling (notice that the order $1<3<5<2<4$ gives $123,345,234$ which is not a shelling). However, the following theorem shows that any unit interval order on a pure strongly connected complex induces a lex shelling.

\begin{theorem} \label{thm:unitintervallex}
Let $\Delta$ be a pure strongly connected $d$-dimensional simplicial complex and assume $[n]$ is a unit interval order on its vertex set. This order induces a lex shelling on $\Delta.$ 
\end{theorem}

\begin{proof}
Let $F_1, F_2, \dots , F_m$ be the facets of $\Delta$ ordered lexicographically. Let $k \in [m-1]$ and define $\Delta_k = \ideal{F_1, \dots, F_k}.$ Assume that $F_1, \dots, F_k$ is a shelling of $\Delta_k$, 
and let $F_{k+1} = v_0 v_1 \dots v_d$ with $v_0 < v_1 < \dots < v_d$. We will show that adding $F_{k+1}$ is a shelling step.

We first consider the case where $F_{k+1}$ has a gap, i.e., there exists an $x \in [n]$ and an $i \in [d]$ such that $v_{i-1} < x < v_{i}$. Further assume that there are no lower gaps in $F_{k+1}$, i.e., that $v_j = v_0 + j$ for all $j < i.$ Let $\sigma \in \ideal{F_{k+1}} \cap \Delta_k$ with $\dim \sigma < d-1$. We will show that $\sigma$ is contained in some $\tau$ such that $\tau \in \ideal{F_{k+1}} \cap \Delta_k$ and $\dim \tau = d-1,$ which will establish the desired result in this case.

Observe, by the unit interval property, that $v_0 v_1 \dots v_{i-1} x \boldsymbol{v}$ is a facet of $\Delta$ for all $\boldsymbol{v} \subseteq \set{v_i,v_{i+1},\dots,v_d}$ with $\abs{\boldsymbol{v}} = d-i-1$. (In other words, $\abs{\set{v_i,v_{i+1},\dots,v_d} \setminus \boldsymbol{v}} = 1.$) Furthermore, each of these facets is lexicographically before $F_{k+1}$, so they are facets of $\Delta_k.$ If $\sigma$ does not contain all vertices $v_i,v_{i+1},\dots,v_d$, then one of these facets will also contain $\sigma$. Intersecting this facet with $F_{k+1}$ gives the desired $\tau$.

Instead, assume that $\sigma$ contains all vertices $v_i,v_{i+1},\dots,v_d$. Thus, there exist vertices $y_0 < y_1 < \dots < y_{\ell}$ such that $E = y_0 y_1 \dots y_{\ell} \sigma$ is a facet of $\Delta_k.$ This implies that $y_0 \le v_0$. We claim that $y_0 = v_0$ produces a contradiction: Assume this is the case. Since $E < F_{k+1}$ and $i$ is the lowest gap in $F_{k+1}$, we see that $E$ contains the vertices $v_0,v_1,\dots v_i$, hence $E = F_{k+1}$. This implies that $y_0 < v_0.$ Thus, by the unit interval property, $y_0 v_1 v_2 \dots v_d$ is a facet of $\Delta$ and hence $\Delta_k$, so $\tau = v_1 v_2 \dots v_d$ contains $\sigma$.

Consider now the case where $F_{k+1} = v_0v_1\dots v_{d}$ is gap-free, that is, $v_j = v_0 + j$ for all $j \in [d]$. 

Let $v_\ell$ be the largest vertex of $F_{k+1}$ which is also in $\Delta_k$. We first claim that if $\ell \ge d-1$, then adding $F_{k+1}$ is a shelling step.

To see this, look at the case $\ell=d-1$. Thus $v_{d-1} \in \Delta_k$ and $v_{d} \notin \Delta_k.$ Since $v_d \notin \Delta_k$, there is a facet $x_0 x_1 \dots x_{d-1} v_{d-1} \in \Delta_k$ with $x_{i-1} < x_{i}$ for $i \in [d-1]$. Since $F_{k+1}$ is gap-free, we see that $x_0 < v_0.$ Thus, the unit interval property shows that $x_0 v_0 \dots v_{d-1} \in \Delta$ and, furthermore, $x_0 v_0 \dots v_{d-1} \in \Delta_k$ by the lexicographic order on the facets of $\Delta$. Since $v_{d} \notin \Delta_k$, this shows that $\langle F_{k+1}\rangle \cap \Delta_k = \langle v_0 v_1 \dots v_{d-1} \rangle$, so adding $F_{k+1}$ is a shelling step.

We instead consider if ${\ell} = d$. By the same logic of the previous paragraph, we see that there exist vertices $x_i$ such that $x_0 \dots x_{d-1} v_{d} \in \Delta$ with $x_0 < v_0$. The unit interval property and the lexicographic order on the facets together show that $x_0 \boldsymbol{v} \in \Delta_k$ for all $\boldsymbol{v} \subseteq \{v_0 , v_1 , \dots , v_{d}\}$ with $|\boldsymbol{v}|=d$. Thus, $\langle F_{k+1}\rangle \cap \Delta_k$ is the complete $(d-1)$-skeleton of $\langle F_{k+1} \rangle$, so again adding $F_{k+1}$ is a shelling step.

We assume now that ${\ell} < d-1$ and we will deduce a contradiction. Since ${\ell} < d-1$, we see that $\dim (\langle F_{k+1}\rangle \cap \Delta_k) < d-1$. Since $\Delta$ is strongly connected, we can find a shortest path from $F_{k+1}$ to $\Delta_k$ in the dual graph of $\Delta$. 

Let $x_0 x_1 \dots x_{d}$ be the last facet in this path (hence in $\Delta_k$) and let $y x_1 \dots x_{d}$ be the previous facet in this path. We know that $y x_1 \dots x_{d} \ne F_{k+1}$ because the dimension of $\langle F_{k+1}\rangle \cap \Delta_k$ is lower than $d-1$.

Thus we see that
\begin{equation}
    x_0 x_1 \dots x_{d} < v_0 v_1 \dots v_{d} < y x_1 \dots x_{d}
\end{equation}
in lexicographic order. So $\min \{ x_0, x_1, \dots, x_{d} \} < v_0$ and $\min \{y, x_1, \dots , x_{d} \} \ge v_0$ because $F_{k+1}$ is the lowest possible facet in lexicographic order with minimum element $v_0.$ Thus $\min \{ x_0, x_1, \dots, x_{d} \} = x_0.$ Without loss of generality, we assume that $x_1 < x_2 < \dots < x_{d}$ and observe that $v_0 \le x_1$.

However, we see that $x_{d} \ge v_{d}$ because $F_{k+1}$ is gap-free. 

Thus the unit interval property and the lexicographic order on the facets together imply that $x_0\boldsymbol{v} \in \Delta_k$ for all $\boldsymbol{v} \subseteq \{v_0 , v_1 , \dots , v_{d}\}$ with $|\boldsymbol{v}|=d$. In particular, we see that $x_0 v_0 v_1 \dots v_{d-1} \in \Delta_k$. But this contradicts the assumption that ${\ell} < d-1$. Therefore the case where ${\ell} < d-1$ cannot occur.
\end{proof}

In dimension one, we are able to weaken the unit interval property to interval and still induce a lex shelling. However, in higher dimensions, this is no longer the case.

\begin{theorem}\label{thm:underclosed}

 Let $\Delta$ be a pure strongly connected $d$-dimensional simplicial complex. If $d=1$, then any interval  order on the vertices of $\Delta$ induces a lex shelling. If $d>1$, there exist complexes with interval labelings that are not shellable.

\end{theorem}
\begin{proof}

    For $d=1$, assume $\Delta$ is a connected graph with an interval (i.e., interval) order on its vertices. We can use an analogous argument as in the proof of Theorem \ref{thm:unitintervallex} to show that this vertex order induces a lex shelling on $\Delta$:

    Assume that we have a partial lex shelling of $\Delta$ and let $\Delta'$ be the complex generated by this partial shelling. Further assume that we are attempting to add the edge $jk$ with $j < k$. To show that adding $jk$ is a shelling step, we only need to guarantee that at least one of $j$ and $k$ is contained in $\Delta'.$

    Assume that neither $j$ nor $k$ is in $\Delta'.$ Since $\Delta$ is connected, there exists a path from $\Delta'$ to the edge $jk$. Let $\ell$ be the last vertex in this path that is contained in $\Delta'$. If $\ell < j$, then all edges containing $\ell$ would be contained in $\Delta'$, so $j < \ell.$ Thus $i \ell \in \Delta'$ for some $i<j.$ Since the order on the vertices of $\Delta$ is interval, this implies that $ij \in \Delta.$ Observe that $ij < jk$, which implies that $ij \in \Delta'$, which is a contradiction. Thus at least one of $j,k$ is contained in $\Delta'$ and therefore adding $jk$ is a shelling step. 

    For $d\geq 2$, consider the following complex $\Delta$ with 5 facets. 
    \begin{align*}
    \Delta = \langle &12 \cdots(d-1)d(d+1), \\ 
                     &12\cdots (d-1)d(d+2), \\
                     &12\cdots (d-1)d(d+3), \\
                     &23\cdots d(d+1)(d+2), \\
                     &34\cdots (d+1)(d+2)(d+3)  \rangle
    \end{align*}
    
    Observe that the link of the vertex labeled $(d+3)$ is 
    \begin{align*}
    \link_\Delta(d+3) = \ideal{12\cdots (d-1)d, 34\cdots (d+1)(d+2)}.
    \end{align*}
    
    It is well known that links in shellable complexes are themselves shellable. The above link is not shellable (notice that its two facets share only a codimension $2$ face), and therefore the whole complex is not shellable.
\end{proof}

However, even in dimension one, we cannot further weaken the vertex ordering to a semi-closed order and be guaranteed a lex shelling, as shown in the following example.

\begin{example} \label{ex:semi_closed_not_lex}
    Consider the graph on 4 vertices that is a triangle with a leaf. Consider a labeling on its vertices so that its edges are 14, 23, 24, 34. This labeling is semi-closed, but does not induce lex shelling.    
\end{example}

\section{Vertex decomposability}

A strictly stronger notion than shellability is vertex decomposability.

\begin{definition}
    A pure simplicial complex $\Delta$ is \emph{vertex decomposable} if
    \begin{compactenum}[(i)]
        \item either $\Delta$ is a simplex or $\Delta = \{\varnothing\}$,
        \item\label{def:shed} or there exists a vertex $v$ such that $\link_\Delta(v)$ and $\del_\Delta(v)$ are vertex decomposable, and every facet of $\del_\Delta(v)$ is a facet of $\Delta$.
    \end{compactenum}
    A vertex $v$ described in \eqref{def:shed} is called a \textit{shedding vertex}.
\end{definition}

We can verify that a complex is vertex decomposable by giving a \emph{shedding order}, i.e., a list of its vertices so that each one is a shedding vertex after deleting the vertex before it. Perhaps unsurprisingly, there exist vertex decomposable complexes with no lex shellings.

\begin{example}\label{ex:point_above}
    The complex $\Delta$ in Example~\ref{ex:VD_no_lex_shellings} is shellable but wiht no lex shellings, as mentioned above. It is also vertex decomposable. One shedding order is $6,5,4,3,1,0,7.$ 
\end{example}

In fact, lex shellability and vertex decomposability are incomparable. There also exist complexes with lex shellings that are not vertex decomposable.

\begin{example}\label{ex:Hachimori}\cite[Section~5.3]{Ha00}
    The following complex is Hachimori's well-known example of a two-dimensional complex that is shellable but not vertex decomposable (or extendably shellable). We have relabeled the vertices so that the standard order $1<2<\dots<7$ induces a lex shelling.
    $$
    \Delta = \ideal{124,125,135,136,147,167,235,236,246,346,347,357,567}
    $$
    To recover Hachimori's original labeling, one can use the map
    $$1 \mapsto 5 ~~~~ 2 \mapsto 6 ~~~~ 3 \mapsto 2 ~~~~ 4 \mapsto 7 ~~~~ 5 \mapsto 3 ~~~~ 6 \mapsto 1 ~~~~ 7 \mapsto 4.$$ We note in passing that of the $7!$ orderings on the vertices of this complex, only $7$ induce a lex shelling.
\end{example}

It turns out that for strongly connected complexes, the unit interval property is enough to guarantee vertex decomposability as well.

\begin{theorem}\label{thm:unitintervaldecomposable}
Let $\Delta$ be a  pure strongly connected $d$-dimensional simplicial complex and assume $[n]$ is a unit interval order on its vertex set. Then $\Delta$ is vertex decomposable.
\end{theorem}

\begin{proof}

We induct on number of vertices. Assume $\Delta$ is not a simplex and let $n$ be the largest vertex of $\Delta$. We will show that $n$ is a shedding vertex.

First consider $\del_{\Delta}(n).$ This complex is unit interval under the original vertex ordering (see \cite[Lemma~58]{BSV}). Since $\Delta$ is strongly connected and unit interval, $\Delta$ contains the facets $H_1,H_2,\dots,H_{n-d}$ where $H_i$ is the gap-free $d$-dimensional facet with lowest vertex $i$ \cite[Theorem~56]{BSV}. Thus $\del_{\Delta}(n)$ will contain the facets $H_1,H_2,\dots,H_{n-d-1}$. Therefore $\del_{\Delta}(n)$ is also strongly connected \cite[Theorem~56]{BSV}. Since $\del_{\Delta}(n)$ is unit interval and strongly connected with $n-1$ vertices, it is vertex decomposable by induction.

We now consider $\link_{\Delta}(n).$ Let $j$ be the smallest vertex that appears in a facet of $\Delta$ containing $n$. Since $\Delta$ is unit interval, this means that $\link_{\Delta}(v)$ is the $(d-1)$-skeleton of the simplex on vertex set $\set{j,j+1,\dots,n-1},$ 
which is known to be vertex decomposable. Thus $n$ is a shedding vertex.

\end{proof}

We recall the definition of shelling completable, which was introduced in \cite{CDGO}: If there exists a shelling of $\Delta$ (a complex on $n$ vertices) that can be taken as the initial sequence of some shelling of the $d$-dimensional skeleton of a simplex on $n$ vertices, we call $\Delta$ \textit{shelling completable}. We immediately get the following corollary due to \cite[Corollary~2.11]{CDGO}, which shows that pure vertex decomposable complexes are shelling completable. 

\begin{corollary}\label{cor:unitintervalcompletable}
Let $\Delta$ be a  pure strongly connected simplicial complex and assume that $[n]$ is a unit interval order on its vertex set. Then $\Delta$ is shelling completable.
\end{corollary}

Simon's Conjecture \cite{Si94} posits that the $d$-skeleton of the simplex on $n$ vertices is extendably shellable; this is equivalent to the claim that all shellable complexes are in fact shelling completable \cite{CDGO}. Simon's Conjecture is known to be true when $d \le 2$ as proved in \cite{BE94} (which actually showed that all rank $3$ matroid independence complexes are extendably shellable, something that is not true in higher ranks) and $d \ge n-3$ proved independently by \cite{BYZ19} and \cite{Do21}.

A complex can have lex shelling orders without satisfying any of the aforementioned vertex orders. 

\begin{example}\label{ex:matroid_not_semi_closed}
Let $\Delta = \ideal{123,125,234,245}.$ This complex is a matroid independence complex, thus every vertex order gives a lex shelling, and the complex is vertex decomposable. However, this complex has no semi-closed vertex orders (cf.~\cite[Lemma~46]{BSV}).
\end{example}

\section*{Acknowledgments}

We thank Bruno Benedetti, Anton Dochtermann, Joseph Doolittle, Florian Frick, Alexander Lazar, and Lisa Seccia for helpful conversations. We also thank Joseph Doolittle and Alexander Lazar for computational help and allowing the use of Examples~\ref{ex:VD_no_lex_shellings} and \ref{ex:Hachimori}.


\end{document}